\documentclass[11pt]{amsproc}
\usepackage{amsfonts}

\theoremstyle{plain}

\newtheorem{definition}{Definition}

\newtheorem{lemma}{Lemma}

\newtheorem{proposition}{Proposition}

\newtheorem{theorem}{Theorem}
\numberwithin{equation}{section}
\input{tcilatex}

\begin{document}
\title{Variational properties of $\sigma _{u}-$curvature for closed
submanifolds of arbitrary codimension in Riemannian manifolds}
\author{Mohammed Benalili}
\address{Dept. Maths, Facult\'{e} des Sciences, Universit\'{e} UABBT Tlemcen
Alg\'{e}rie.\\
m\_benalili@yahoo.fr}
\maketitle

\begin{abstract}
The objet of this paper is the study of variations of a functional whose
integrant is the $\sigma _{u}-$ curvature of closed submanifolds of
arbitrary codimension in Riemannian manifolds.
\end{abstract}

\section{Introduction}

The study of Riemannian geometry is based on the analysis of geometric
operators like the shape operator, the Ricci tensor, the Schouten operator
etc...Some functions built from these operators play a fundamental role in
the understanding of this discipline, in particular algebraic invariants
such the $r$-th symmetric functions $\sigma _{r}$ associated with the shape
operator and the Newton transformations $T_{r}$.The following articles can
be consulted on this subject (\cite{M.A1}, \cite{M.A2}, \cite{Case}, \cite%
{Cis}, \cite{K.N1}, \cite{K.N2}, \cite{Reilly}, \cite{Reilly2}, \cite{Rund}, 
\cite{An}). Reilly ( see \cite{Reilly} ) has considered the variations where
the integrand of the functional is a function of the $r$-th mean curvatures $%
\sigma _{r}$ of a hypersurfaces in a space form of constant curvature. In a
recent paper Case (see \cite{Case}) introduced and studied the notion of \ $r
$-th weighted curvature which generalizes the mean curvature. In this
context, we have extended some of Reilly's results still in the framework of
hypersurfaces of the space form (\cite{Ben}). For the study of variations of
submanifolds of a fixed Riemannian manifold in the sense of Reilly (\cite%
{Reilly}) and to my knowledge nothing was done in the more general case of
the higher codimension. In fairly recent and nice paper (\cite{An}), K.
Andrzejewski, W. Kozlowski and K. Niedzialomski introduced and developed an
algebraic framework which adapts very well to the case of higher codimension
submanifolds which gave me the idea of developing a work in the higher
codimension similar to that of Reilly (\cite{Reilly}). To be more precise,
let $A$ be an endomorphism of a $m$-dimensional real vector space $V$ and
denote by $End(V)$ the space of endomorphisms on $V.$ Let $N(q)$ stands for
the set of all $q$-uplets $u=\left( u_{1},...,u_{q}\right) $ where $u_{j}$
stand for positive integers. Denote by $End^{q}(V)=End(V)\times ...\times
End(V)$, $q$-times. For any $A=(A_{1},...,A_{q})\in End^{q}(V)$, $%
u=(u_{1},...,u_{q})\in N(q)$, $t=(t_{1},...,t_{q})\in R^{q}$, define:%
\begin{equation*}
tA=t_{1}A_{1}+...+t_{q}A_{q}
\end{equation*}%
and 
\begin{equation*}
t^{u}=(t^{u_{1}},...,t^{u_{q}})
\end{equation*}%
the Newton polynomial of $A$ is then given by%
\begin{eqnarray*}
P_{A}(t) &=&\det (tA+id_{V}) \\
&=&\dsum\limits_{\left\vert u\right\vert \leq p}\sigma _{u}(A)t^{u}\text{.}
\end{eqnarray*}%
Now if $\psi _{t}:M^{m}\rightarrow \overline{M}^{n}$ is a one parameter
family of immersions of an $m-$dimensional closed manifold $M^{m}$ into an $n
$-Riemannian manifold $\left( \overline{M}^{n}\text{,}\left\langle
,\right\rangle \right) $ denote $\left( \nu _{1},...,\nu _{q=n-m}\right) $ a
local normal basis to $M^{n}.$ If $A_{i}=A_{\nu _{i}}$ denotes the shape
operator with respect to the normal local vector field $\nu _{i}$, $i=1,...,q
$, put $A=(A_{1},...,A_{q})$. For any $x\in M$, $\ A(x)\in End^{q}(T_{x}M)$,
where $T_{x}M$ \ stands for the tangent space in $x$ to $M^{m}$. $\ $So $%
\sigma _{u}$ is $u$-th elementary symmetric function. The aim of this paper
is the study of the variations of the functional $\int_{M^{m}}\sigma _{u}dV$%
. We start by deriving the first formula for the first variation of our
functional, first in a general context which is formulated in Theorem\ref%
{theorem1} of the section tree, then in the particular case of space with
constant curvature given by Theorem\ref{theorem2} of the section four. Some
applications to submanifolds of the Euclidian space and the unit round
sphere are given.

\section{Generalized Newton Transformations}

Let $A$ be an endomorphism of a $m$-dimensional real vector space $V$
endowed with the usual inner product$.$ The Newton's transformations
associated with $A$ is a family $T=\left( T_{r}\right) _{r\in N}$ defined
recurrently by:

\begin{equation*}
T_{o}=id_{V}
\end{equation*}%
\begin{equation*}
T_{r}=\sigma _{r}id_{V}-AT_{r-1}.
\end{equation*}%
Denote by $End(V)$ the space of endomorphisms on $V$ . Let $A\in End(V)$ and 
$A^{\ast }$ its adjoint endomorphism$.$ $End(V)$ is then endowed with the
inner product $\left\langle A,B\right\rangle =tr(AB^{\ast })$ where $A$, $%
B\in End(V)$. Borrowing notations from the paper (\cite{An}), we denote by $%
N(q)$ the set of all $q$-uplets $u=\left( u_{1},...,u_{q}\right) $ where $%
u_{j}$ stand for positive integers. Let $End^{q}(V)=End(V)\times ...\times
End(V)$, $q$-times. For any $A=(A_{1},...,A_{q})\in End^{q}(V)$, $%
u=(u_{1},...,u_{q})\in N(q)$, $t=(t_{1},...,t_{q})\in R^{q}$. Define:%
\begin{equation*}
tA=t_{1}A_{1}+...+t_{q}A_{q}
\end{equation*}%
and 
\begin{equation*}
t^{u}=(t^{u_{1}},...,t^{u_{q}})
\end{equation*}%
the Newton polynomial of $A$ is then given by%
\begin{eqnarray*}
P_{A}(t) &=&\det (tA+id_{V}) \\
&=&\dsum\limits_{\left\vert u\right\vert \leq p}\sigma _{u}(A)t^{u}\text{.}
\end{eqnarray*}%
Consider the musical functions $\alpha ^{\#},\alpha _{b}:N(q)\rightarrow N(q)
$ given by%
\begin{equation*}
\alpha ^{\#}(u_{1},...,u_{q})=(u_{1},...,u_{\alpha -1},u_{\alpha
}+1,u_{\alpha +1},...,u_{q})
\end{equation*}%
and%
\begin{equation*}
\alpha _{b}(u_{1},...,u_{q})=(u_{1},...,u_{\alpha -1},u_{\alpha
}-1,u_{\alpha +1},...,u_{q}).
\end{equation*}%
The generalized Newton transformations (in abbreviated form GNT) are defined
( see \cite{An}) by: for any curve $t\rightarrow A(t)$ in End$^{d}$(V) such
that $A(0)=A$ the GNT of $A$ is a family of endomorphisms $\left(
T_{u}\right) _{u\in N(q)}$ given by 
\begin{equation}
\frac{d}{dt}\sigma _{u}(t)\mid _{t=0}=\sum_{\alpha }tr(\frac{d}{dt}A_{\alpha
}(t)\mid _{t=0}).T_{\alpha _{b}(u)}).  \tag{1}  \label{1}
\end{equation}%
Once again we use the notations of ( \cite{An}). For $i=(i^{1},...,i^{q})\in
N(s,q)$, its weight is defined as $\left\vert i\right\vert =\left(
\left\vert i^{1}\right\vert ,...,\left\vert i^{q}\right\vert \right) \in N(q)
$ and its length by $\left\Vert i\right\Vert =\sum_{\alpha =1}^{q}\left\vert
i^{\alpha }\right\vert =\sum_{\alpha =1,j=1}^{q,s}i_{j}^{\alpha }.$ Denote
by $I(q,s)$ the subset of $N(q,s)$ of matrices $i$ such that:

-each entry of $i$ is either $0$ or $1$

-the length of $i$ is $s$

-every column of $i$ contains only one entry equal to $1$.

Let $A$ $=(A_{1},...,A_{q})\in End^{q}(V)$ and $i\in I(q,s)$ with $I(q,0)$
is the set of vector $0,$ we put ( as in )%
\begin{equation*}
A^{i}=A_{1}^{i_{1}^{1}}A_{2}^{i_{1}^{2}}...A_{q}^{i_{1}^{q}}...A_{1}^{i_{s}^{1}}A_{2}^{i_{s}^{2}}...A_{q}^{i_{s}^{q}}
\end{equation*}

with 
\begin{equation*}
A^{0}=1_{V}\text{.}
\end{equation*}%
We quote after (\cite{An}):

\begin{proposition}
\label{prop1}The generalized Newton transformations $\left( T_{u}:u\in
N(q)\right) $ of $A=(A_{1},...,A_{q})$ enjoy the following fundamental
properties:

(1) For every $u\in N(q)$ with $\left\vert u\right\vert \geq m$, $T_{u}=0$

(2) Symmetric functions $\sigma _{u}$ are given by

\begin{equation*}
\left\vert u\right\vert \sigma _{u}=\dsum\limits_{\alpha }tr(A_{\alpha
}T_{\alpha _{b}(u)}).
\end{equation*}

(3) The trace of $T_{u}$ expresses as 
\begin{equation*}
tr(T_{u})=\left( m-\left\vert u\right\vert \right) \sigma _{u}.
\end{equation*}

(4) Symmetric functions $\sigma _{u}$ fulfill the following recurrence
relation%
\begin{equation*}
\sum_{\alpha ,\beta }tr(A_{\alpha }A_{\beta }T_{\beta _{b}\alpha
_{b}(u)})=-\left\vert u\right\vert \sigma _{u}+\sum_{\beta }tr\left(
A_{\beta }\right) \sigma _{\beta _{b}(u)}.
\end{equation*}

(5) 
\begin{equation*}
T_{u}=\sum_{s=0}^{\left\vert u\right\vert }\sum_{i\in I(q,s)}\left(
-1\right) ^{\left\Vert i\right\Vert }\sigma _{u-\left\vert i\right\vert
}A^{i}.
\end{equation*}

(6) 
\begin{eqnarray*}
T_{u} &=&\sigma _{u}1_{V}-\sum_{\alpha }A_{\alpha }T_{\alpha _{b}(u)} \\
&=&\sigma _{u}1_{V}-\sum_{\alpha }T_{\alpha _{b}(u)}A_{\alpha }
\end{eqnarray*}%
with $\left\vert u\right\vert \geq 1$.
\end{proposition}

\section{Variational properties}

\ \ Consider a family of one parameter $\ \psi _{t}:M^{m}\rightarrow 
\overline{M}^{n}$ of immersions of an $m-$dimensional closed manifold $M^{m}$
into an $n$-Riemannian manifold $\left( \overline{M}^{n}\text{,}\left\langle
,\right\rangle \right) $. We consider on $M^{m}$ the Riemannian metric
induced by the metric on $\overline{M}^{n}$. If $\overline{\nabla }$ stands
for the covariant derivative in $\overline{M}^{n}$, for every vector fields $%
X,Y$ tangent along $M^{m}$ in a neighborhood of a point $x$, the Gauss
formula writes as:%
\begin{equation*}
\overline{\nabla }_{X}Y=\nabla _{X}Y+\alpha (X,Y)
\end{equation*}%
where

$\nabla $ is the induced covariant derivative $M^{m}$ defined by $\nabla
_{X}Y=\left( \overline{\nabla }_{X}Y\right) ^{\top }$

$\alpha $ is the second fundamental form given by $\alpha \left( X,Y\right)
=\left( \overline{\nabla }_{X}Y\right) ^{\bot }.$

Similarly if $\nu $ is a normal vector field along $M^{m}$ in a neighborhood
of $x,$we obtain the Weirgarten equation:%
\begin{equation*}
\overline{\nabla }_{X}\nu =-A_{\nu }\left( X\right) +D_{X}\nu
\end{equation*}%
where

$D$ denotes the covariant derivative on the normal bundle of $M,$ defined by
: $D_{X}\nu =\left( \overline{\nabla }_{X}\nu \right) ^{\bot }$

$A_{\nu }$ is the shape operator $A_{\nu }(X)=-\left( \overline{\nabla }%
_{X}\nu \right) ^{\top }$ which is related to the second fundamental form by:%
\begin{equation*}
\left\langle A_{\nu }(X),Y\right\rangle =\left\langle \alpha \left(
X,Y\right) ,\nu \right\rangle 
\end{equation*}%
for any vector fields $X$, $Y$ on $M.$ So if $\left( \nu _{1},...,\nu
_{q=n-m}\right) $ is a normal local basis to $M^{m}$ and $A_{i}=A_{\nu _{i}}$%
, $i=1,...,q$, stands for the shape operator with respect to $\nu _{i}$ we
let $A=\left( A_{1},...,A_{q}\right) $ and denote by $\sigma _{u}$ the $u$%
-th elementary symmetric function.

Consider the following variational problem%
\begin{equation*}
\delta \left( \int_{M}\sigma _{u}dV\right) =0
\end{equation*}%
with $u\in I(q,s)$.

\begin{theorem}
\label{theorem1} With the above notations and assumptions the first
variation of the global $\sigma _{u}$-curvature is given by:%
\begin{eqnarray*}
&&\frac{d}{dt}\left( \int_{M^{m}}\sigma _{u}dV\right) \\
&=&\int_{M^{m}}\left\{ -g^{jk}\sum_{\alpha }R^{\overline{M}}(\nu ^{\alpha },%
\frac{\partial \psi }{\partial x_{k}},\frac{\partial \psi }{\partial x_{i}}%
,X)\left( T_{\alpha _{b}(u)}\right) _{j}^{i}+\left( T_{\alpha
_{b}(u)}\right) ^{ij}\lambda _{\alpha },_{ij}\right. \\
&&+\left. g^{jk}\left( R^{\overline{M}}\right) ^{\bot }(\frac{\partial \psi 
}{\partial x_{i}},X^{\bot })\frac{\partial \psi }{\partial x_{k}}\left(
T_{\alpha _{b}(u)}\right) \right. \\
&&+g^{jk}\left\langle \frac{\partial \nu ^{\alpha }}{\partial t},\nu ^{\beta
}\right\rangle \left( A_{\beta }\right) _{ik}\left( T_{\alpha
_{b}(u)}\right) _{j}^{i}-g^{jk}\lambda _{\beta ,k}\left\langle D_{\frac{%
\partial }{\partial x_{i}}}\nu ^{\alpha },\nu ^{\beta }\right\rangle \left(
T_{\alpha _{b}(u)}\right) _{j}^{i} \\
&&-\lambda _{\beta }g^{jk}\left\langle D_{\frac{\partial }{\partial x_{i}}%
}\nu ^{\alpha },\nu ^{\gamma }\right\rangle \left\langle D_{\frac{\partial }{%
\partial x_{k}}}\nu ^{\beta },\nu ^{\gamma }\right\rangle \left( T_{\alpha
_{b}(u)}\right) _{j}^{i} \\
&&\left. -\mu ^{m}g^{jk}\left\langle D_{\frac{\partial }{\partial x_{i}}}\nu
^{\alpha },\overline{\nabla }_{\frac{\partial }{\partial x_{k}}}\frac{%
\partial \psi }{\partial x_{m}}\right\rangle \left( T_{\alpha
_{b}(u)}\right) _{j}^{i}-\left\langle \lambda ,\beta ^{\#}\left( u\right)
\right\rangle \sigma _{\beta ^{\#}(u)}\right\} dV.
\end{eqnarray*}
\end{theorem}

By definition of $\sigma _{u}$, we have%
\begin{equation*}
\frac{\partial \sigma _{u}}{\partial t}=\sum_{\alpha }\text{tr}\left( \frac{%
\partial A_{\alpha }}{\partial t}T_{\alpha _{b}(u)}\right)
\end{equation*}%
with%
\begin{equation*}
\text{tr}\left( \frac{\partial A_{\alpha }}{\partial t}T_{\alpha _{b}\left(
u\right) }\right) =\frac{\partial \left( A_{\alpha }\right) _{j}^{i}}{%
\partial t}\left( T_{\alpha _{b}(u)}\right) _{i}^{j}
\end{equation*}%
and%
\begin{equation*}
\left( A_{\alpha }\right) _{i}^{j}=g^{jk}\left( A_{\alpha }\right) _{ik}
\end{equation*}%
where 
\begin{equation*}
\left( A_{\alpha }\right) _{ik}=\left\langle \overline{\nabla }_{\frac{%
\partial }{\partial x_{i}}}\frac{\partial \psi }{\partial x_{k}}),\nu
_{\alpha }\right\rangle =-\left\langle \overline{\nabla }_{\frac{\partial
\psi }{\partial x_{i}}}\nu _{\alpha },\frac{\partial \psi }{\partial x_{k}}%
\right\rangle
\end{equation*}%
where $\left( \nu ^{1},...,\nu ^{k}\right) $ is an orthogonal basis to $%
M^{m} $ and $k=n-m$.

Hence 
\begin{equation*}
\frac{\partial \left( A_{\alpha }\right) _{i}^{j}}{\partial t}=\frac{%
\partial g^{jk}}{\partial t}\left( A_{\alpha }\right) _{ik}+g^{jk}\frac{%
\partial \left( A_{\alpha }\right) _{ik}}{\partial t}\text{.}
\end{equation*}%
Obviously%
\begin{equation*}
\frac{\partial g^{jk}}{\partial t}=-g^{jl}\frac{\partial g_{pl}}{\partial t}%
g^{pk}
\end{equation*}%
Now, if we consider the calculations in a normal coordinates that is at a
point $x\in M$ \ where the metric tensor fulfills $g_{ij}(x)=\left\langle 
\frac{\partial \psi }{\partial x_{i}},\frac{\partial \psi }{\partial x_{j}}%
\right\rangle =\delta _{ij}$ and $\Gamma _{ij}^{k}(x)=0$, where $\Gamma
_{ij}^{k}$ stand for the Christoffel symbols corresponding to the metric
connection $\nabla $ on $M$, we get%
\begin{eqnarray*}
\frac{\partial g_{pl}}{\partial t} &=&\left\langle \overline{\nabla }_{\frac{%
\partial }{\partial t}}\frac{\partial \psi }{\partial x_{p}},\frac{\partial
\psi }{\partial x_{l}}\right\rangle +\left\langle \frac{\partial \psi }{%
\partial x_{p}},\overline{\nabla }_{\frac{\partial }{\partial t}}\frac{%
\partial \psi }{\partial x_{l}}\right\rangle \\
&=&\left\langle \overline{\nabla }_{\frac{\partial }{\partial x_{p}}}X,\frac{%
\partial \psi }{\partial x_{l}}\right\rangle +\left\langle \frac{\partial
\psi }{\partial x_{p}},\overline{\nabla }_{\frac{\partial }{\partial x_{l}}%
}X\right\rangle \\
&=&\left\langle \overline{\nabla }_{\frac{\partial }{\partial x_{p}}}\left(
\lambda _{\beta }\nu ^{\beta }+\mu ^{m}\frac{\partial \psi }{\partial x_{m}}%
\right) ,\frac{\partial \psi }{\partial x_{l}}\right\rangle \\
&&+\left\langle \frac{\partial \psi }{\partial x_{p}},\overline{\nabla }_{%
\frac{\partial }{\partial x_{l}}}\left( \lambda _{\beta }\nu ^{\beta }+\mu
^{m}\frac{\partial \psi }{\partial x_{m}}\right) \right\rangle \\
&=&\mu _{p,l}+\mu _{l,p}-2\lambda _{\beta }\left( A_{\beta }\right) _{pl}.
\end{eqnarray*}%
hence%
\begin{equation}
\frac{\partial g^{jk}}{\partial t}=-g^{jl}g^{pk}\left( \mu _{p,l}+\mu
_{l,p}-2\lambda _{\beta }\left( A_{\beta }\right) _{pl}\right) .  \tag{10}
\label{10}
\end{equation}

We have also%
\begin{eqnarray}
\frac{\partial \nu ^{\alpha }}{\partial t} &=&\left\langle \frac{\partial
\nu ^{\alpha }}{\partial t},\frac{\partial \psi }{\partial x_{k}}%
\right\rangle \frac{\partial \psi }{\partial x_{k}}+\left\langle \frac{%
\partial \nu ^{\alpha }}{\partial t},\nu ^{\beta }\right\rangle \nu ^{\beta }
\TCItag{11}  \label{11} \\
&=&-\left\langle \nu ^{\alpha },\overline{\nabla }_{\frac{\partial }{%
\partial t}}\frac{\partial \psi }{\partial x_{k}}\right\rangle \frac{%
\partial \psi }{\partial x_{k}}+\left\langle \frac{\partial \nu ^{\alpha }}{%
\partial t},\nu ^{\beta }\right\rangle \nu ^{\beta }  \notag \\
&=&-\left\langle \nu ^{\alpha },\overline{\nabla }_{\frac{\partial }{%
\partial x_{k}}}X\right\rangle \frac{\partial \psi }{\partial x_{k}}%
+\left\langle \frac{\partial \nu ^{\alpha }}{\partial t},\nu ^{\beta
}\right\rangle \nu ^{\beta }  \notag \\
&=&-g^{jk}\left( \lambda _{\alpha },_{j}+\mu ^{l}\left( A_{\alpha }\right)
_{jl}\right) \frac{\partial \psi }{\partial x_{k}}+\left\langle \frac{%
\partial \nu ^{\alpha }}{\partial t},\nu ^{\beta }\right\rangle \nu ^{\beta
}.  \notag
\end{eqnarray}%
Now we compute%
\begin{eqnarray*}
\frac{\partial \left( A_{\alpha }\right) _{ik}}{\partial t} &=&-\left\langle 
\overline{\nabla }_{\frac{\partial }{\partial t}}\overline{\nabla }_{\frac{%
\partial }{\partial x_{i}}}\nu ^{\alpha },\frac{\partial \psi }{\partial
x_{k}}\right\rangle -\left\langle \overline{\nabla }_{\frac{\partial }{%
\partial x_{i}}}\nu ^{\alpha },\overline{\nabla }_{\frac{\partial }{\partial
t}}\frac{\partial \psi }{\partial x_{k}}\right\rangle \\
&=&-R^{\overline{M}}(\nu ^{\alpha },\frac{\partial \psi }{\partial x_{k}},%
\frac{\partial \psi }{\partial x_{i}},X)-\left\langle \overline{\nabla }_{%
\frac{\partial }{\partial x_{i}}}\overline{\nabla }_{\frac{\partial }{%
\partial x_{t}}}\nu ^{\alpha },\frac{\partial \psi }{\partial x_{k}}%
\right\rangle \\
&&-\left\langle \overline{\nabla }_{\frac{\partial }{\partial x_{i}}}\nu
^{\alpha },\overline{\nabla }_{\frac{\partial }{\partial x_{k}}%
}X\right\rangle .
\end{eqnarray*}%
By formula (\ref{11}), we get%
\begin{eqnarray*}
\overline{\nabla }_{\frac{\partial }{\partial x_{i}}}\overline{\nabla }_{%
\frac{\partial }{\partial x_{t}}}\nu ^{\alpha } &=&-g^{jm}\left( \left(
\lambda _{\alpha },_{ji}+\mu ^{l,i}\left( A_{\alpha }\right) _{jl}+\mu
^{l}\left( A_{\alpha }\right) _{jl_{,i}}\right) \frac{\partial \psi }{%
\partial x_{m}}\right) \\
&&-g^{jk}\left( \lambda _{\alpha },_{j}+\mu ^{l}\left( A_{\alpha }\right)
_{jl}\right) \overline{\nabla }_{\frac{\partial }{\partial x_{i}}}\frac{%
\partial \psi }{\partial x_{k}} \\
&&+\frac{\partial }{\partial x_{i}}\left\langle \frac{\partial \nu ^{\alpha }%
}{\partial t},\nu ^{\beta }\right\rangle \nu ^{\beta }+\left\langle \frac{%
\partial \nu ^{\alpha }}{\partial t},\nu ^{\beta }\right\rangle \overline{%
\nabla }_{\frac{\partial }{\partial x_{i}}}\nu ^{\beta }
\end{eqnarray*}%
so%
\begin{eqnarray*}
\left\langle \overline{\nabla }_{\frac{\partial }{\partial x_{i}}}\overline{%
\nabla }_{\frac{\partial }{\partial x_{t}}}\nu ^{\alpha },\frac{\partial
\psi }{\partial x_{k}}\right\rangle &=&-g^{jm}\left( \lambda _{\alpha
},_{ji}+\mu ^{l,i}\left( A_{\alpha }\right) _{jl}+\mu ^{l}\left( A_{\alpha
}\right) _{jl_{,i}}\right) g_{mk} \\
&&+\left\langle \frac{\partial \nu ^{\alpha }}{\partial t},\nu ^{\beta
}\right\rangle \left\langle \overline{\nabla }_{\frac{\partial }{\partial
x_{i}}}\nu ^{\beta },\frac{\partial \psi }{\partial x_{k}}\right\rangle \\
&=&-\left( \lambda _{\alpha },_{ki}+\mu ^{l,i}\left( A_{\alpha }\right)
_{kl}+\mu ^{l}\left( A_{\alpha }\right) _{kl_{,i}}\right) -\left\langle 
\frac{\partial \nu ^{\alpha }}{\partial t},\nu ^{\beta }\right\rangle \left(
A_{\beta }\right) _{ik}.
\end{eqnarray*}%
In the same manner, we have 
\begin{eqnarray*}
\left\langle \overline{\nabla }_{\frac{\partial }{\partial x_{i}}}\nu
^{\alpha },\overline{\nabla }_{\frac{\partial }{\partial x_{k}}%
}X\right\rangle &=&\left\langle \overline{\nabla }_{\frac{\partial }{%
\partial x_{i}}}\nu ^{\alpha },\overline{\nabla }_{\frac{\partial }{\partial
x_{k}}}\left( \lambda _{\beta }\nu ^{\beta }+\mu ^{m}\frac{\partial \psi }{%
\partial x_{m}}\right) \right\rangle \\
&=&\left\langle \overline{\nabla }_{\frac{\partial }{\partial x_{i}}}\nu
^{\alpha },\lambda _{\beta ,k}\nu ^{\beta }+\lambda _{\beta }\overline{%
\nabla }_{\frac{\partial }{\partial x_{k}}}\nu ^{\beta }+\mu ^{m,k}\frac{%
\partial \psi }{\partial x_{m}}+\mu ^{m}\overline{\nabla }_{\frac{\partial }{%
\partial x_{k}}}\frac{\partial \psi }{\partial x_{m}}\right\rangle \\
&=&\lambda _{\beta ,k}\left\langle \overline{\nabla }_{\frac{\partial }{%
\partial x_{i}}}\nu ^{\alpha },\nu ^{\beta }\right\rangle +\lambda _{\beta
}\left\langle \overline{\nabla }_{\frac{\partial }{\partial x_{i}}}\nu
^{\alpha },\overline{\nabla }_{\frac{\partial }{\partial x_{k}}}\nu ^{\beta
}\right\rangle -\mu ^{m,k}\left( A_{\alpha }\right) _{im} \\
&&+\mu ^{m}\left\langle \overline{\nabla }_{\frac{\partial }{\partial x_{i}}%
}\nu ^{\alpha },\overline{\nabla }_{\frac{\partial }{\partial x_{k}}}\frac{%
\partial \psi }{\partial x_{m}}\right\rangle .
\end{eqnarray*}%
By noticing that%
\begin{eqnarray*}
\left\langle \overline{\nabla }_{\frac{\partial }{\partial x_{i}}}\nu
^{\alpha },\overline{\nabla }_{\frac{\partial }{\partial x_{k}}}\nu ^{\beta
}\right\rangle &=&\left\langle \overline{\nabla }_{\frac{\partial }{\partial
x_{i}}}\nu ^{\alpha },\frac{\partial \psi }{\partial x_{j}}\right\rangle
\left\langle \overline{\nabla }_{\frac{\partial }{\partial x_{k}}}\nu
^{\beta },\frac{\partial \psi }{\partial x_{j}}\right\rangle +\left\langle 
\overline{\nabla }_{\frac{\partial }{\partial x_{i}}}\nu ^{\alpha },\nu
^{\gamma }\right\rangle \left\langle \nu ^{\gamma },\overline{\nabla }_{%
\frac{\partial }{\partial x_{k}}}\nu ^{\beta }\right\rangle \\
&=&\left( A_{\alpha }A_{\beta }\right) _{ik}+\left\langle D_{\frac{\partial 
}{\partial x_{i}}}\nu ^{\alpha },\nu ^{\gamma }\right\rangle \left\langle D_{%
\frac{\partial }{\partial x_{k}}}\nu ^{\beta },\nu ^{\gamma }\right\rangle
\end{eqnarray*}%
where $D$ stands for the connection on the normal fiber bundle.

Consequently%
\begin{eqnarray*}
\left\langle \overline{\nabla }_{\frac{\partial }{\partial x_{i}}}\nu
^{\alpha },\overline{\nabla }_{\frac{\partial }{\partial x_{k}}%
}X\right\rangle &=&\lambda _{\beta ,k}\left\langle D_{\frac{\partial }{%
\partial x_{i}}}\nu ^{\alpha },\nu ^{\beta }\right\rangle +\lambda _{\beta
}\left( A_{\alpha }A_{\beta }\right) _{ik} \\
&&+\lambda _{\beta }\left\langle D_{\frac{\partial }{\partial x_{i}}}\nu
^{\alpha },\nu ^{\gamma }\right\rangle \left\langle D_{\frac{\partial }{%
\partial x_{k}}}\nu ^{\beta },\nu ^{\gamma }\right\rangle \\
&&-\mu _{,k}^{m}\left( A_{\alpha }\right) _{im}+\mu ^{m}\left\langle D_{%
\frac{\partial }{\partial x_{i}}}\nu ^{\alpha },\overline{\nabla }_{\frac{%
\partial }{\partial x_{k}}}\frac{\partial \psi }{\partial x_{m}}%
\right\rangle .
\end{eqnarray*}

We get%
\begin{eqnarray*}
g^{jk}\frac{\partial \left( A_{\alpha }\right) _{ik}}{\partial t}
&=&-g^{jk}R^{\overline{M}}(\nu ^{\alpha },\frac{\partial \psi }{\partial
x_{k}},\frac{\partial \psi }{\partial x_{i}},X)+g^{jk}\left( \lambda
_{\alpha },_{ki}+\mu ^{l,i}\left( A_{\alpha }\right) _{kl}+\mu ^{l}\left(
A_{\alpha }\right) _{kl_{,i}}\right) \\
&&+g^{jk}\left\langle \frac{\partial \nu ^{\alpha }}{\partial t},\nu ^{\beta
}\right\rangle \left( A_{\beta }\right) _{ik}-g^{jk}\lambda _{\beta
,k}\left\langle D_{\frac{\partial }{\partial x_{i}}}\nu ^{\alpha },\nu
^{\beta }\right\rangle -g^{jk}\lambda _{\beta }\left( A_{\alpha }A_{\beta
}\right) _{ik} \\
&&-\lambda _{\beta }g^{jk}\left\langle D_{\frac{\partial }{\partial x_{i}}%
}\nu ^{\alpha },\nu ^{\gamma }\right\rangle \left\langle D_{\frac{\partial }{%
\partial x_{k}}}\nu ^{\beta },\nu ^{\gamma }\right\rangle \\
&&+g^{jk}\mu _{,k}^{m}\left( A_{\alpha }\right) _{im}-\mu
^{m}g^{jk}\left\langle D_{\frac{\partial }{\partial x_{i}}}\nu ^{\alpha },%
\overline{\nabla }_{\frac{\partial }{\partial x_{k}}}\frac{\partial \psi }{%
\partial x_{m}}\right\rangle .
\end{eqnarray*}%
Taking into account formula (\ref{10}), we get%
\begin{eqnarray}
\frac{\partial \left( A_{\alpha }\right) _{i}^{j}}{\partial t} &=&-g^{jk}R^{%
\overline{M}}(\nu ^{\alpha },\frac{\partial \psi }{\partial x_{k}},\frac{%
\partial \psi }{\partial x_{i}},X)+g^{jk}\left( \lambda _{\alpha },_{ki}+\mu
^{l}\left( A_{\alpha }\right) _{kl_{,i}}\right)  \notag \\
&&+g^{jk}\left\langle \frac{\partial \nu ^{\alpha }}{\partial t},\nu ^{\beta
}\right\rangle \left( A_{\beta }\right) _{ik}-g^{jk}\lambda _{\beta
,k}\left\langle D_{\frac{\partial }{\partial x_{i}}}\nu ^{\alpha },\nu
^{\beta }\right\rangle  \notag \\
&&-\lambda _{\beta }g^{jk}\left\langle D_{\frac{\partial }{\partial x_{i}}%
}\nu ^{\alpha },\nu ^{\gamma }\right\rangle \left\langle D_{\frac{\partial }{%
\partial x_{k}}}\nu ^{\beta },\nu ^{\gamma }\right\rangle  \TCItag{12}
\label{12} \\
&&-\mu ^{m}g^{jk}\left\langle D_{\frac{\partial }{\partial x_{i}}}\nu
^{\alpha },\overline{\nabla }_{\frac{\partial }{\partial x_{k}}}\frac{%
\partial \psi }{\partial x_{m}}\right\rangle +\lambda _{\beta }\left(
A_{\alpha }A_{\beta }\right) _{i}^{j}.  \notag
\end{eqnarray}%
To compute $\frac{\partial \sigma _{u}}{\partial t}$ we multiply both sides
of (\ref{12}) by $\left( T_{\alpha _{b}(u)}\right) _{j}^{i}$ and sum.

For the continuation of the calculations we will need the following lemma
which is a form of the property (4) of the proposition (\ref{prop1}).

\begin{lemma}
\label{lem1} For any $\lambda =(\lambda _{1},...,\lambda _{m})\in R^{m}$,
the symmetric functions $\sigma _{u}$ fulfill the following recurrence
relation:%
\begin{equation*}
\sum_{\alpha ,\beta }\lambda _{\beta }\text{tr}(A_{\alpha }A_{\beta
}T_{\beta _{b}\alpha _{b}(u)})=-\left\langle \lambda ,u\right\rangle \sigma
_{u}+\sum_{\beta }\lambda _{\beta }\text{tr}\left( A_{\beta }\right) \sigma
_{\beta _{b}(u)}
\end{equation*}%
where \ $\left\langle \lambda ,u\right\rangle =\sum_{\alpha }\lambda
_{\alpha }u_{\alpha }.$
\end{lemma}

\begin{proof}
First, we have 
\begin{eqnarray*}
\sum_{\beta }\lambda _{\beta }A_{\beta }T_{\beta _{b}\left( u\right) }
&=&\sum_{\beta }\lambda _{\beta }A_{\beta }\left( \sigma _{\beta _{b}\left(
u\right) }\right) 1_{V}-\sum_{\alpha }A_{\alpha }T_{\alpha _{b}\beta
_{b}\left( u\right) }) \\
&=&\sum_{\beta }\lambda _{\beta }\sigma _{\beta _{b}\left( u\right)
}A_{\beta }-\sum_{\alpha ,\beta }\lambda _{\beta }A_{\beta }A_{\alpha
}T_{\alpha _{b}\beta _{b}\left( u\right) }.
\end{eqnarray*}%
By passing to the traces, 
\begin{equation*}
\sum_{\alpha ,\beta }\lambda _{\beta }tr\left( A_{\beta }A_{\alpha
}T_{\alpha _{b}\beta _{b}\left( u\right) }\right) =\sum_{\beta }\lambda
_{\beta }tr(A_{\beta })\sigma _{\beta _{b}\left( u\right) }-\sum_{\beta
}\lambda _{\beta }tr\left( A_{\beta }T_{\beta _{b}\left( u\right) }\right) .
\end{equation*}%
It remains to compute the last term of the right hand side of the above
equality. To do so, we consider the curve $A(\tau )=(1_{V}+\lambda \tau )A$.
We get $A(0)=A$ and$\frac{dA\left( \tau \right) }{d\tau }\left\vert \tau
=0\right. =\lambda A.$ The expending of the polynomial $P_{A\left( \tau
\right) }\left( t\right) =\det (1_{V}+tA\left( \tau \right) )=\sum_{u}%
\widetilde{\sigma }_{u}\left( \tau \right) t^{u}$ with $\widetilde{\sigma }%
_{u}\left( \tau \right) =(1+\lambda \tau )^{u}\sigma _{u}$. We have%
\begin{equation*}
\frac{d\widetilde{\sigma }_{u}\left( \tau \right) }{d\tau }\left\vert \tau
=0\right. =\sum_{\beta }u_{\beta }\lambda _{\beta }\sigma _{u}=\left\langle
\lambda ,u\right\rangle \sigma _{u}
\end{equation*}%
and 
\begin{equation*}
\sum_{\beta }tr(\frac{d}{dt}A_{\beta }(\tau )\mid _{\tau =0}.T_{\beta
_{b}(u)})=\sum_{\beta }\lambda _{\beta }tr(A_{\beta }T_{\alpha _{\beta
}(u)}).
\end{equation*}%
By the definition of the (GTN) Newton transformations we obtain%
\begin{equation*}
\sum_{\beta }\lambda _{\beta }tr(A_{\beta }T_{\alpha _{\beta
}(u)})=\left\langle \lambda ,u\right\rangle \sigma _{u}\text{.}
\end{equation*}
\end{proof}

First, by Lemma (\ref{lem1}), we have

\begin{eqnarray}
\sum_{\alpha ,\beta }\lambda _{\beta }\left( A_{\alpha }A_{\beta }\right)
_{i}^{j}\left( T_{\alpha _{b}(u)}\right) _{j}^{i} &=&\sum_{\alpha ,\beta
}\lambda _{\beta }\text{tr}\left( A_{\alpha }A_{\beta }T_{\alpha
_{b}(u)}\right)  \notag \\
&=&\sum_{\alpha ,\beta }\lambda _{\beta }\text{tr}\left( A_{\alpha }A_{\beta
}T_{\beta _{b}\beta ^{\#}\alpha _{b}(u)}\right)  \TCItag{13}  \label{13} \\
&=&\sum_{\alpha ,\beta }\lambda _{\beta }\text{tr}\left( A_{\alpha }A_{\beta
}T_{\beta _{b}\alpha _{b}\beta ^{\#}(u)}\right)  \notag \\
&=&-\left\langle \lambda ,\beta ^{\#}(u)\right\rangle \sigma _{\beta
^{\#}(u)}+\sum_{\beta }\lambda _{\beta }\text{tr}\left( A_{\beta }\right)
.\sigma _{u}  \notag
\end{eqnarray}%
we also write, 
\begin{equation}
g^{jk}\sum_{\alpha }\lambda _{\alpha },_{ki}\left( T_{\alpha _{b}(u)}\right)
_{j}^{i}=\sum_{\alpha }\lambda _{\alpha },_{ki}\left( T_{\alpha
_{b}(u)}\right) ^{ki}  \tag{14}  \label{14}
\end{equation}%
By the Codazzi formula,we get 
\begin{eqnarray*}
g^{jk}\mu ^{l}\sum_{\alpha }\left( A_{\alpha }\right) _{kl},_{i}\left(
T_{\alpha _{b}(u)}\right) _{j}^{i} &=&g^{jk}\mu ^{l}\sum_{\alpha }\left(
A_{\alpha }\right) _{ki},_{l}\left( T_{\alpha _{b}(u)}\right) _{j}^{i} \\
&&+g^{jk}\left( R^{\overline{M}}\right) ^{\bot }(\frac{\partial \psi }{%
\partial x_{i}},X^{\bot })\frac{\partial \psi }{\partial x_{k}}\left(
T_{\alpha _{b}(u)}\right) _{j}^{i}
\end{eqnarray*}%
and by equation (\ref{1}), we infer that%
\begin{eqnarray}
g^{jk}\mu ^{l}\sum_{\alpha }\left( A_{\alpha }\right) _{kl},_{i}\left(
T_{\alpha _{b}(u)}\right) _{j}^{i} &=&\sum_{\alpha }\mu ^{l}\text{tr}\left(
\left( A_{\alpha }\right) ,_{l}\left( T_{\alpha _{b}(u)}\right) \right) 
\notag \\
&&+g^{jk}\sum_{\alpha }\left( R^{\overline{M}}\right) ^{\bot }(\frac{%
\partial \psi }{\partial x_{i}},X^{\bot })\frac{\partial \psi }{\partial
x_{k}}\left( T_{\alpha _{b}(u)}\right) _{j}^{i}  \TCItag{16}  \label{16} \\
&=&\mu ^{l}\sigma _{u,l}+g^{jk}\left( R^{\overline{M}}\right) ^{\bot }(\frac{%
\partial \psi }{\partial x_{i}},X^{\bot })\frac{\partial \psi }{\partial
x_{k}}\left( T_{\alpha _{b}(u)}\right) _{j}^{i}.  \notag
\end{eqnarray}%
Hence 
\begin{eqnarray}
\frac{\partial \sigma _{u}}{\partial t} &=&-g^{jk}\sum_{\alpha }R^{\overline{%
M}}(\nu ^{\alpha },\frac{\partial \psi }{\partial x_{k}},\frac{\partial \psi 
}{\partial x_{i}},X)\left( T_{\alpha _{b}(u)}\right) _{j}^{i}+\sum_{\alpha
}\left( T_{\alpha _{b}(u)}\right) ^{ij}\lambda _{\alpha },_{ij}  \notag \\
&&+\mu ^{l}\sigma _{u,l}+g^{jk}\left( R^{\overline{M}}\right) ^{\bot }(\frac{%
\partial \psi }{\partial x_{i}},X^{\bot })\frac{\partial \psi }{\partial
x_{k}}\left( T_{\alpha _{b}(u)}\right) _{j}^{i}  \notag \\
&&+g^{jk}\left\langle \frac{\partial \nu ^{\alpha }}{\partial t},\nu ^{\beta
}\right\rangle \left( A_{\beta }\right) _{ik}\left( T_{\alpha
_{b}(u)}\right) _{j}^{i}-g^{jk}\lambda _{\beta ,k}\left\langle D_{\frac{%
\partial }{\partial x_{i}}}\nu ^{\alpha },\nu ^{\beta }\right\rangle \left(
T_{\alpha _{b}(u)}\right) _{j}^{i}  \TCItag{17}  \label{17} \\
&&-\lambda _{\beta }g^{jk}\left\langle D_{\frac{\partial }{\partial x_{i}}%
}\nu ^{\alpha },\nu ^{\gamma }\right\rangle \left\langle D_{\frac{\partial }{%
\partial x_{k}}}\nu ^{\beta },\nu ^{\gamma }\right\rangle \left( T_{\alpha
_{b}(u)}\right) _{j}^{i}  \notag \\
&&-\mu ^{m}g^{jk}\left\langle D_{\frac{\partial }{\partial x_{i}}}\nu
^{\alpha },\overline{\nabla }_{\frac{\partial }{\partial x_{k}}}\frac{%
\partial \psi }{\partial x_{m}}\right\rangle \left( T_{\alpha
_{b}(u)}\right) _{j}^{i}-\left\langle \lambda ,u\right\rangle
\dsum\limits_{\beta }\lambda _{\beta }\sigma _{\beta ^{\#}(u)}  \notag \\
&&+\sum_{\beta }\lambda _{\beta }\text{tr}\left( A_{\beta }\right) \sigma
_{u}.  \notag
\end{eqnarray}%
The expression of $\frac{\partial dV}{\partial t}$is standard and it is
given by:%
\begin{equation}
\frac{\partial dV}{\partial t}=\left( -\lambda _{\alpha }\text{tr(}A_{\alpha
}\text{)}+\mu _{,l}^{l}\right) dV.  \tag{18}  \label{18}
\end{equation}%
By expressions (\ref{17}) and (\ref{18}), we infer that:%
\begin{equation*}
\frac{\partial \sigma _{u}}{\partial t}+\sigma _{u}\left( -\lambda _{\alpha }%
\text{tr(}A_{\alpha }\text{)}+\mu _{,l}^{l}\right) =
\end{equation*}%
\begin{eqnarray}
&&-g^{jk}\sum_{\alpha }R^{\overline{M}}(\nu ^{\alpha },\frac{\partial \psi }{%
\partial x_{k}},\frac{\partial \psi }{\partial x_{i}},X)\left( T_{\alpha
_{b}(u)}\right) _{j}^{i}+\sum_{\alpha }\text{tr}\left( T_{\alpha _{b}(u)}%
\text{hess}(\lambda _{\alpha }\right) )  \notag \\
&&+g^{jk}\left( R^{\overline{M}}\right) ^{\bot }(\frac{\partial \psi }{%
\partial x_{i}},X^{\bot })\frac{\partial \psi }{\partial x_{k}}\left(
T_{\alpha _{b}(u)}\right) _{j}^{i}  \notag \\
&&+g^{jk}\left\langle \frac{\partial \nu ^{\alpha }}{\partial t},\nu ^{\beta
}\right\rangle \left( A_{\beta }\right) _{ik}\left( T_{\alpha
_{b}(u)}\right) _{j}^{i}-g^{jk}\lambda _{\beta ,k}\left\langle D_{\frac{%
\partial }{\partial x_{i}}}\nu ^{\alpha },\nu ^{\beta }\right\rangle \left(
T_{\alpha _{b}(u)}\right) _{j}^{i}  \TCItag{19}  \label{19} \\
&&-\lambda _{\beta }g^{jk}\left\langle D_{\frac{\partial }{\partial x_{i}}%
}\nu ^{\alpha },\nu ^{\gamma }\right\rangle \left\langle D_{\frac{\partial }{%
\partial x_{k}}}\nu ^{\beta },\nu ^{\gamma }\right\rangle \left( T_{\alpha
_{b}(u)}\right) _{j}^{i}  \notag \\
&&-\mu ^{m}g^{jk}\left\langle D_{\frac{\partial }{\partial x_{i}}}\nu
^{\alpha },\overline{\nabla }_{\frac{\partial }{\partial x_{k}}}\frac{%
\partial \psi }{\partial x_{m}}\right\rangle \left( T_{\alpha
_{b}(u)}\right) _{j}^{i}-\left\langle \lambda ,\beta ^{\#}(u)\right\rangle
\sigma _{\beta ^{\#}(u)}  \notag \\
&&+\left( \sigma _{u}\mu ^{l}\right) _{,l}.  \notag
\end{eqnarray}%
where hess is the hessian of $\lambda _{\alpha }.$

\section{Special cases}

To simplify the expression of the integrand in Theorem \ref{theorem1}, we
consider submanifod with flat normal bundle. First we recall the following
facts: for $x\in M$ the tangent space\ $T_{x}\left( \overline{M}^{n}\right) $
splits as:

\begin{equation*}
T_{x}\left( \overline{M}^{n}\right) =T_{x}\left( M^{m}\right) \oplus
N_{x}\left( M^{m}\right)
\end{equation*}%
where $T_{x}\left( M^{m}\right) $ is the tangent space of $M^{m}$ at $x$ and 
$N_{x}\left( M^{m}\right) =T_{x}\left( M^{m}\right) ^{\bot }$ the normal
space of $M^{m}$ at $x$.

Let $D$ denote the normal covariant derivative on a $m$-dimensional
submanifold $M^{n}$ of a Riemannian manifold $\overline{M}^{n}$ and consider
the curvature tensor of the normal bundle%
\begin{equation*}
R_{D}(X,Y)\nu =D_{X}D_{Y}\nu -D_{Y}D_{X}\nu -D_{\left[ X;Y\right] }\nu .
\end{equation*}%
The the normal bundle $N(M^{m})$ of $M^{n}$ in $\overline{M}^{n}$ is said
flat if and only if $R_{D}(x)=0$ for any $x\in M^{n}$ and $M^{n}$ is called
submanifold with flat normal bundle. The normal connection is called flat if
the normal bundle of $M^{n}$is flat. It is well known in this case there is
in each point $y$ of $\overline{M}^{n}$ an orthonormal basis $\left( \nu
^{1},...,\nu ^{n-m}\right) $ of $N(M^{m})$ such that each vector field $\nu
^{\alpha }$ is parallel in $N(M^{m})$ that is to say $\overline{\nabla }%
_{X}\nu ^{\alpha }=0$ for each $\nu ^{\alpha }\in N(M^{m})$ and $X\in
T(M^{m}).$ If the ambient manifold $\overline{M}^{n}$ has a constant
curvature $c$ then for any vector fields $\overline{X},\overline{Y},%
\overline{Z},$ the curvature tensor of $\overline{M}^{n}$ is given by 
\begin{equation*}
R^{\overline{M}^{n}}(\overline{X},\overline{Y})\overline{Z}=c\left(
\left\langle \overline{Z},\overline{Y}\right\rangle \overline{X}%
-\left\langle \overline{Z},\overline{X}\right\rangle \overline{Y}\right)
\end{equation*}%
so for $X$,$Y$ tangent and $\nu $ normal to $M^{n}$%
\begin{equation}
R_{D}(X,Y)\nu =0.  \tag{20}  \label{20}
\end{equation}%
As a consequence of formula (\ref{20}), we have

\begin{theorem}
\label{theorem2} Let $M^{m}$ be an $m$-dimensional closed submanifold of an $%
n$-dimensional space $\overline{M}$ $^{n}(c)$ of constant sectional
curvature $c.$ The first variation of the global $\sigma _{u}$-curvature is
given by:%
\begin{eqnarray*}
&&\frac{d}{dt}\left( \int_{M^{m}}\sigma _{u}dV\right) \\
&=&\int_{M^{m}}\left( -\left\langle \lambda ,\beta ^{\#}\left( u\right)
\right\rangle \sigma _{\beta ^{\#}(u)}+c\left( m+1-\left\vert u\right\vert
\right) \sum_{\alpha }\lambda _{\alpha }\sigma _{\alpha _{b}(u)}\right) dV
\end{eqnarray*}
\end{theorem}

\begin{proof}
Indeed, we have:

\begin{eqnarray*}
g^{jk}\sum_{\alpha }R^{\overline{M}}(\nu ^{\alpha },\frac{\partial \psi }{%
\partial x_{k}},\frac{\partial \psi }{\partial x_{i}},X)\left( T_{\alpha
_{b}(u)}\right) _{j}^{i} &=&-c\left\langle \frac{\partial \psi }{\partial
x_{i}},\frac{\partial \psi }{\partial x_{k}}\right\rangle \sum_{\alpha
}\left\langle X,v^{\alpha }\right\rangle \left( T_{\alpha _{b}(u)}\right)
_{j}^{i} \\
&=&-cg^{jk}g_{ik}\sum_{\alpha ,\beta }\lambda _{\beta }\left\langle \nu
^{\beta },v^{\alpha }\right\rangle \left( T_{\alpha _{b}(u)}\right) _{j}^{i}
\\
&=&-c\left( m+1-\left\vert u\right\vert \right) \sum_{\alpha }\lambda
_{\alpha }\sigma _{\alpha _{b}(u)}.
\end{eqnarray*}%
On the other hand since the normal connection of $M^{n}$ is flat, for every $%
x\in M^{m}$ there exist an orthonormal vector fields $\nu ^{1},...\nu ^{n-m}$
in an open neighborhood $U$ of $x$ such $D_{Y}\nu =0$ in $U$ where $Y\in
T_{x}M$. Let $c_{s}$ be a curve on $\overline{M}^{n}$ such that $\nu
^{\alpha }=\left. \frac{\partial }{\partial s}\right\vert _{s=0}c_{s}$. Then 
$\left. \frac{\partial \nu ^{\alpha }}{\partial t}\right\vert _{t=0}=\left. 
\frac{\partial ^{2}}{\partial t\partial s}\right\vert _{t=s=0}c_{s}\left(
\psi _{t}\right) =\overline{\nabla }_{X}\nu ^{\alpha }$ and $\left\langle
\left. \frac{\partial \nu ^{\alpha }}{\partial t}\right\vert _{t=0},\nu
^{\beta }\right\rangle =\left\langle \overline{\nabla }_{X}\nu ^{\alpha
},\nu ^{\beta }\right\rangle .$ So Since the integral of a differentiable
form $\omega $ on a manifold $M$ is defined as the sum of integrals of this
form multiplied by an element $\rho _{i}$ of a partition $\left( \rho
_{i}\right) _{i\in I}$ subordinated to an open cover $\left( U_{j}\right)
_{j\in J}$ of $M$ $\ $over $U_{j(i)}$ which contains the support of $\rho
_{i}$; it follows that the integrals of all terms containing the normal
covariant derivatives in the expression (\ref{19}) cancel. The same is also
true for $\lambda _{\alpha },_{ij}$ since in an open neighborhood $U$ of
each point $x\in M^{n}$, we have 
\begin{equation*}
\lambda _{\alpha },_{i}=\left\langle \overline{\nabla }_{\frac{\partial }{%
\partial x_{i}}}X^{\bot },\nu ^{\alpha }\right\rangle +\left\langle X^{\bot
},\overline{\nabla }_{\frac{\partial }{\partial x_{i}}}\nu ^{\alpha
}\right\rangle =0
\end{equation*}%
hence $\lambda _{\alpha },_{ij}=0$ in $U.$ Consequently the integral on $%
M^{n}$ of $\left( T_{\alpha _{b}(u)}\right) ^{ij}\lambda _{\alpha },_{ij}$
cancels also.
\end{proof}

\subsection{Submanifolds of Euclidean space}

\begin{definition}
A submanifold $M^{m}$ of an Euclidean space $E^{n}$ is said $\sigma _{u}$%
-minimal if $\sigma _{v}$ vanishes identically where $v\in N(n-m)$ is a
multi-index with length $\left\vert v\right\vert =\left\vert u\right\vert +1$%
.
\end{definition}

As in the paper of Reilly (see \cite{Reilly}) we will express the minimality
of a submanifold of an Euclidean space $E^{n}$in terms of partial
differential equations. Let $\psi =(\psi _{1},...,\psi _{n+1})$ be the
position vector of the submanifold $M^{m}$ and $\psi _{,ij}=(\psi
_{1,ij},...,\psi _{n,ij})$ the second covariant derivative of $x$ on $M^{m}$.

\begin{eqnarray*}
\psi _{,ij} &=&\frac{\partial ^{2}\psi }{\partial x_{i}\partial x_{j}}-d\psi
(\nabla _{\frac{\partial }{\partial x_{i}}}\frac{\partial }{\partial x_{j}})
\\
&=&\frac{\partial ^{2}\psi }{\partial x_{i}\partial x_{j}}-\nabla _{\psi
_{\ast }\frac{\partial }{\partial x_{i}}}\psi _{\ast }\frac{\partial }{%
\partial x_{j}} \\
&=&\left\langle \frac{\partial ^{2}\psi }{\partial x_{i}\partial x_{j}}%
,N\right\rangle N \\
&=&\sum_{\alpha }\lambda _{\alpha }\left( A_{\alpha }\right) _{ij}N
\end{eqnarray*}%
where $N$ denotes a normal vector field to $M^{m}$, $\lambda _{\alpha
}=\left\langle N,\nu ^{\alpha }\right\rangle $ and $\left( \nu ^{1},...,\nu
^{n-m}\right) $ an orthonormal basis to $M^{m}$.

Hence 
\begin{equation*}
\psi _{,ij}\left( T_{u}\right) ^{ij}=\sum_{\alpha }\lambda _{\alpha }\left(
A_{\alpha }\right) _{ij}\left( T_{\alpha _{b}\alpha ^{\#}(u)}\right)
^{ij}N=\left\langle \lambda ,\alpha ^{\#}(u)\right\rangle \sigma _{\alpha
^{\#}(u)}N.
\end{equation*}

\subsection{Submanifolds of the unit round sphere}

\begin{definition}
A submanifold in the unit round sphere is said $\sigma _{u}$-minimal with if%
\begin{equation*}
\left\langle \lambda ,\beta ^{\#}\left( u\right) \right\rangle \sigma
_{\beta ^{\#}(u)}-\left( m+1-\left\vert u\right\vert \right) \sum_{\alpha
}\lambda _{\alpha }\sigma _{\alpha _{b}(u)}=0
\end{equation*}
\end{definition}

Let $\psi =(\psi _{1},...,\psi _{n})$ be the position vector of the
hypersurface $M^{m}$ in the unit round sphere $S^{n}$ and $\psi _{,ij}=(\psi
_{1,ij},...,\psi _{n,ij})$ the second covariant derivative of $x$ on $M^{m}$%
. If $\nabla $ denotes the covariant derivative on $M^{n}$ induced by the
covariants derivative $\nabla ^{S^{n}}$ on the unit round sphere. We have 
\begin{eqnarray}
\psi _{,ij} &=&\nabla _{\frac{\partial }{x_{j}}}\nabla _{\frac{\partial }{%
\partial x_{i}}}\psi =\left( \nabla _{\frac{\partial }{x_{j}}}d\psi \right) (%
\frac{\partial }{\partial x_{i}})  \notag \\
&=&\nabla _{\frac{\partial }{x_{j}}}d\psi (\frac{\partial }{\partial x_{i}}%
)-d\psi \left( \nabla _{\frac{\partial }{x_{j}}}\frac{\partial }{\partial
x_{i}}\right)  \notag \\
&=&\left\langle \nabla _{\frac{\partial }{x_{j}}}^{S^{n}}d\psi (\frac{%
\partial }{\partial x_{i}}),N\right\rangle N  \notag \\
&=&\dsum\limits_{\alpha }\lambda _{\alpha }\left\langle \nabla _{\frac{%
\partial }{x_{j}}}^{S^{n}}d\psi (\frac{\partial }{\partial x_{i}}),\nu
^{\alpha }\right\rangle N  \TCItag{28}  \label{28} \\
&=&\dsum\limits_{\alpha }\lambda _{\alpha }\left( A_{\alpha }\right) _{ij}N.
\notag
\end{eqnarray}%
where $\nu ^{\alpha }$, $\alpha =1,...,n-m$ is a normal orthonormal basis to 
$M^{m}$, $N$ a normal vector field to $M^{m}$ (as submanifold of $S^{n})$
and $\lambda _{\alpha }=\left\langle N,\nu _{\alpha }\right\rangle $. In
order to characterize the $\sigma _{u}$-minimality of submanifolds of the
unit sphere, in case $\left\langle \lambda ,\beta _{b}\left( u\right)
\right\rangle \neq 0$, we multiply both sides of (\ref{28}) by $%
T_{u}^{ij}-\left( m-\left\vert u\right\vert +1\right) \sum_{\beta }\frac{%
\lambda _{\beta }}{\left\langle \lambda ,\beta _{b}\left( u\right)
\right\rangle }T_{\beta _{b}\alpha _{b}\left( u\right) }^{ij}$ and sum to
infer%
\begin{eqnarray*}
T_{u}^{ij}\psi _{,ij} &=&\dsum\limits_{\alpha }\lambda _{\alpha }T_{\alpha
_{b}\alpha ^{\#}\left( u\right) }^{ij}\left( A_{\alpha }\right) _{ij}-\left(
m-\left\vert u\right\vert +1\right) \sum_{\alpha ,\beta }T_{\alpha _{b}\beta
_{b}\left( u\right) }^{ij}\lambda _{\alpha }\frac{\lambda _{\beta }}{%
\left\langle \lambda ,\beta _{b}\left( u\right) \right\rangle }\left(
A_{\alpha }\right) _{ij} \\
&=&\left\langle \lambda ,\beta ^{\#}\left( u\right) \right\rangle \sigma
_{\beta ^{\#}(u)}-\left( m-\left\vert u\right\vert +1\right) \sum_{\alpha
}\lambda _{\alpha }\sigma _{\alpha _{b}(u)}.
\end{eqnarray*}%
if $\left\langle \lambda ,\beta _{b}\left( u\right) \right\rangle =\lambda
_{1}u_{1}+...+\lambda _{\beta -1}u_{\beta -1}+\lambda _{\beta }\left(
u_{\beta }-1\right) +...+\lambda _{m}u_{m}$ for any multi-index $u$ with
lenght $\left\vert u\right\vert \geq 2$ then $\lambda _{\alpha }=0$ and $%
M^{m}$ is a totally geodesic submanifold of $S^{n}$; if $\left\vert
u\right\vert =1$ necessarily $u_{\beta }=1$ and $M^{m}$ is $\sigma
_{(0,...,0,1,0...0)}$-minimal submanifold of $S^{n}$.

\end{document}